\documentclass[11pt]{amsart}

\usepackage[latin1]{inputenc}
\usepackage[english]{babel}

\usepackage{amsmath, amsthm, amsfonts, amssymb}
\usepackage{graphicx}

\newcommand{\goodchi}{\protect\raisebox{2pt}{$\chi$}}
\newcommand{\abs}   [1]{\left \vert #1 \right \vert}
\newcommand{\enorm} [1]{\left \Vert #1 \right \Vert}
\newcommand{\norm}  [2]{\left \Vert #1 \right \Vert_{#2}}

\newcommand{\diver} [0]{\nabla \cdot}

\makeatletter
\newtheorem*{rep@theorem}{\rep@title}
\newcommand{\newreptheorem}[2]{%
\newenvironment{rep#1}[1]{%
 \def\rep@title{#2 \ref{##1}}%
 \begin{rep@theorem}}%
 {\end{rep@theorem}}}
\makeatother

\newtheorem{theorem}{Theorem}[section]

\newtheorem{lemma}[theorem]{Lemma}
\newtheorem{corollary}[theorem]{Corollary}
\newtheorem{remark}[theorem]{Remark}

\newreptheorem{theorem}{Theorem}

\title[Piecewise smooth recovery]{Reconstruction and stability for piecewise smooth potentials in the plane}
\author{Jorge Tejero}

\begin{document}
\maketitle

\begin{abstract}
We show that complex-valued potentials with jump discontinuities 
can be recovered from the Dirichlet-to-Neumann map 
using Bukhgeim's method. Combining with known formulas, this enables the recovery from the scattering amplitude at a fixed energy. We also provide {\it a priori} stability
estimates for reconstruction from the Dirichlet \!\!-to-Neumann map 
as well as from the scattering amplitude
given an approximate knowledge of the
location of the discontinuities.
\end{abstract}

\section{Introduction}

Let $\Omega$ be a bounded planar domain that contains a bounded potential $V$, and consider the Dirichlet problem
\begin{equation}
\label{dirichletProblem} \left\{
\begin{array}{l}
\Delta u = V u\\
u \vert_{\partial \Omega} = f.
\end{array} \right.
\end{equation}
Supposing that 0 is not a Dirichlet eigenvalue, there is a unique solution $u\in H^1(\Omega)$ and
the DtN map $\Lambda_V$ can be formally defined by
\begin{equation}
\Lambda_V : f \rightarrow \nabla u \cdot \mathbf{n} \vert_{\partial \Omega}. \nonumber
\end{equation}
Our goal is then to recover the potential from 
the information contained in~$\Lambda_V$. This problem has a long history 
and is closely related to the inverse conductivity problem proposed
by Calder\'on; see \cite{C80}. Some relevant work in higher dimensions includes \cite{SU87, Na88, No88, BT03, KSU07, HT, H, CR}. 

The two dimensional question is quite different to the higher dimensional case;
for example the inverse problem is no longer overdetermined.
In \cite{Na96} Nachman introduced the $\overline{\partial}$-method to prove uniqueness 
for conductivities in $W^{2,p}$ with $p>1$,
and gave a reconstruction procedure
(this work has since been extended to more general cases, see for example
\cite{IN95, BU97}). 
In~\cite{AP06}, combining the $\overline{\partial}$-method with the theory of 
quasi-conformal maps, Astala and P\"{a}iv\"{a}rinta solved the uniqueness
problem for $L^\infty$ conductivities.

There was little progress for general potentials in the plane until Bukhgeim  introduced a new method, proving uniqueness for $C^1$-potentials \cite{B08}.
There he took advantage of solutions of the form
\begin{equation}
u = u_{\lambda, x} = e^{i \lambda \psi} (1 + w), \quad \psi(z) = \psi_x (z) = \frac{1}{2}\big(z_1 - x_1 + i (z_2 -  x_2)\big)^2, \nonumber
\end{equation}
where $w=w_{\lambda, x}$ is small in some sense. Indeed, after proving that the 
the integral
\begin{equation}
\int_\Omega e^{i \lambda (\psi + \overline{\psi})} \, V \, w \nonumber
\end{equation}
converges to zero  as $\lambda$ grows, one uses the stationary phase method to see that
the integral
$$ \int_{\partial \Omega} e^{i \lambda \overline{\psi}} (\Lambda_V - \Lambda_0) [u] = \int_\Omega e^{i \lambda (\psi + \overline{\psi})} \, V \, (1 + w) $$
converges to $\pi \lambda^{-1} V(x)$.

This quadratic phase approach
has since been extended by many authors; we list a few here.
Novikov and Santacesaria obtained stability estimates for~$C^2$-potentials in~\cite{NS10}
and proved that the reconstruction method could be extended to $C^1$ matrix-valued potentials in~\cite{NS11}. In \cite{BIY15}, Bl\r{a}sten, Imanuvilov and Yamamoto proved uniqueness
for potentials in $L^p$ with $p>2$, and gave stability estimates in the 
$L^2$ norm for potentials in $H^s$ with $s > 0$.
Finally, Astala, Faraco and Rogers gave a reconstruction
procedure for potentials in~$H^{1/2}$ and
proved that this is best possible in some sense; see \cite{AFR13}.

In the recent work \cite{LNV15}, Lakshtanov, Novikov and 
Vainberg provided a reconstruction scheme for real bounded potentials. 
Their procedure relies on Faddeev's scattering solutions and allows 
to recover almost each potential, in the sense of their Remark 4.1.

Here we will prove that the reconstruction formula given
in~\cite{AFR13} (mildly different from the original Bukhgeim formula) works for {\it piecewise $W^{2,1}$-potentials with
jump discontinuities on smooth curves}. That is to say potentials $V$ that can be written as \begin{equation}
V(x) = \sum_{j=1}^N q_j(x) \goodchi_{\Omega_j} (x), \nonumber
\end{equation}
where $q_j \in W^{2, 1}(\Omega_j)$ and $\Omega_j$ 
are piecewise $C^{2,\alpha}$-domains with $\alpha >1/2$. By this we mean, the boundary $\partial\Omega_j$  can be expressed as a finite union of graphs of $C^{2, \alpha}$-functions and that the union is Lipschitz. The most significant novelties of the article are to be found in Section~\ref{rec}, where we will prove the following theorem. There we will also present a potential for which the recovery formula fails at points away from the discontinuities.

\begin{theorem} \label{reconstructionTheorem0}
Let $V$ be a piecewise $W^{2,1}$-potential with
jump discontinuities on smooth curves. Then
\begin{equation}
\lim_{\lambda \rightarrow \infty } \frac{\lambda}{\pi} \int_{\partial \Omega} e^{i \lambda \overline{\psi}} \left( \Lambda_V - \Lambda_0 \right) \left[ u_{\lambda,x} \right] = V(x),\quad \text{a.e.}\ x\in\Omega. \nonumber
\end{equation}
\end{theorem}

We refer to Theorem~1.1 of \cite{AFR13} for how to determine the values of the Bukhgeim solutions on the boundary.

For a fixed $k>0$, we also consider the Schr\"odinger equation
\begin{equation}
(- \Delta+V) u = k^2 u, \label{scatteringEq}
\end{equation}
where $k^2$ is not a Dirichlet eigenvalue for the Hamiltonian $- \Delta+V $. For $\theta \in \mathbb{S}^1$, the outgoing scattering solutions
satisfy the Lippmann-Schwinger equation
$$ u(x, \theta) = e^{i k x \cdot \theta} - \int_{\mathbb{R}^2} G_0(x,y) V(y) u(y, \theta)\, dy,$$
where $G_0$ denotes the outgoing Green's function which satisfies
$$ (- \Delta - k^2) G_0(x,y) = \delta(x-y). $$
Then the scattering amplitude 
$A_V : \mathbb{S}^1 \times \mathbb{S}^1 \rightarrow \mathbb{C}$ 
at energy $k^2$ can be written
$$ A_V(\eta, \theta) = \int_{\mathbb{R}^2} e^{-i k \eta \cdot y} V(y) u(y, \theta)\, dy.$$
Given an incident plane wave in direction $\theta$, the scattering amplitude 
measures the probability of scattering in the direction $\eta$.
A classical problem is to recover the potential in \eqref{scatteringEq} from
the information contained in $A_V$.

By applying the formula obtained in \cite{AFR14}, we can also recover $V$ from $A_V$ as long as $V$ is a piecewise $W^{2,1}$-potential with
jump discontinuities on smooth curves. More details will be given in the third section.

Stability estimates are a classical theme in inverse problems; see \cite{A14, BIY15} for 
the Schr\"odinger equation and \cite{A88, BBR01, BFR07, CFR10, CGR13} for the conductivity equation. 
Notice that in \cite{BIY15} or in \cite{CFR10} there is stability for discontinuous potentials or conductivities but only in the $L^2$ sense. A careful analysis of the dependence of the constants in the reconstruction theorem yields an $L^\infty$ stability for discontinuous coefficients provided a noise knowledge of the discontinuities of the potential. 
In the final section we will prove stability estimates for the reconstruction from the DtN map and
from the information contained in $A_V$.

\section{Preliminaries}

\subsection{Quadratic phase solutions and integrals}

Using Wirtinger derivatives we can write the the time-independent 
Schr\"odinger equation as $ 4 \partial_z \partial_{\overline{z}}u=Vu$.
Taking solutions of the form $u = e^{i \lambda \psi} (1 + w)$ and multiplying
both sides by $e^{i \lambda \overline{\psi}}$ we obtain
$$ 4 e^{i \lambda \overline{\psi}} \partial_z \partial_{\overline{z}} e^{i \lambda \psi} (1 + w) = e^{i \lambda (\psi + \overline{\psi})} V (1 + w). $$
Taking into account that $ \partial_{\overline{z}} e^{i \lambda \psi} =  \partial_{z} e^{i \lambda \overline{\psi}} = 0 $,  this can be rewritten as
$$ 4 \partial_z e^{i \lambda (\psi + \overline{\psi})} \partial_{\overline{z}} w = e^{i \lambda (\psi + \overline{\psi})} V (1 + w).$$
As the derivatives are local operators, and we need only satisfying the equation inside $\Omega$, we can take solutions of the form
$$ w = \frac{1}{4} \partial_{\overline{z}}^{-1} \left[ e^{-i \lambda (\psi + \overline{\psi})} \, \goodchi_Q \, \partial_z^{-1} \left[ e^{i \lambda (\psi + \overline{\psi})} \, \goodchi_Q \, V \, (1 + w) \right] \right]$$
where $Q$ is an auxiliary axis-parallel square containing $\Omega$.
In order to simplify notation we define the multiplication operators
$$ M^{\pm \lambda} \left[ F \right] =  e^{\pm i \lambda (\psi + \overline{\psi})} \goodchi_Q F , $$
and write
$$ S^\lambda_1 \left[ F \right] = \frac{1}{4} \partial^{-1}_{\overline{z}} \circ M^{- \lambda} \circ  \partial^{-1}_{z} \circ M^{\lambda} , \quad
S^\lambda_V \left[ F \right] =  S^\lambda_1 \left[ V F \right].$$
For sufficiently smooth $V$, the operator norm of $S^\lambda_V$ is small for large enough~$\lambda$; see \cite[Lemma 2.3]{AFR13}, so  we can invert 
$(I - S^\lambda_V)$ using  Neumann series, yielding $$ w = w_{\lambda,x} = (I - S^\lambda_V)^{-1} S^\lambda_1 \left[ V \right]. $$
The following oscillatory integral operator will also be useful in the sequel
\begin{equation}
T^\lambda_w [F] (x) = \frac{\lambda}{\pi} \int_{\mathbb{R}^2} e^{i \lambda (\psi + \overline{\psi})} F(z) w_{\lambda,x}(z)\, dz. \nonumber
\end{equation}
In order to study the behaviour of these operators 
we will use the homogeneous $L^2$ Sobolev  space, 
denoted by $\dot{H}^s$, with norm 
$\norm{f}{\dot{H}^s}=\||\cdot|^s \hat{f}\|_{L^2},$
where $\hat{f}$ is the Fourier transform of $f$. 

We have the following bound for $M^{\pm \lambda}$; the proof can be found in 
\cite[Section 2]{AFR13}. The key ingredient in the proof is the classical lemma of
van der Corput.

\begin{lemma} \label{mNorm}
Let $0 \leq s_1, s_2 < 1$. Then 
$$ \norm{M^{\pm \lambda}[F]}{\dot{H}^{-s_2}} \leq C \lambda^{- \min \{ s_1, s_2 \}} \norm{F}{\dot{H}^{s_1}}, \quad \lambda \geq 1. $$
\end{lemma}

The following two lemmas were essentially proven in \cite[Sections 2 and~4]{AFR13}; 
we present minor modifications, suitable for the stability analysis of the final section.

\begin{lemma} \label{sNorm}
Let $0 < s_1, s_2 < 1$. Then there exists a constant $C$ such that
\begin{equation}
\norm{S^\lambda_1}{\dot{H}^{s_1} \rightarrow \dot{H}^{s_2}} \leq \, C \, \lambda^{-\tau} \nonumber
\end{equation} 
where $ \tau = 1 - s_2 + \min \lbrace s_1, s_2 \rbrace$.
\end{lemma}

\begin{proof}
Using Lemma \ref{mNorm} twice we get
\begin{eqnarray}
\norm{S^\lambda_1}{\dot{H}^{s_1} \rightarrow \dot{H}^{s_2}} & \leq & \norm{M^{-\lambda} \circ \partial_z^{-1} \circ M^\lambda}{\dot{H}^{s_1} \rightarrow \dot{H}^{s_2 - 1}}  \nonumber \\
& \leq & C \, \lambda^{-1 + s_2} \norm{\partial_z^{-1} \circ M^\lambda}{\dot{H}^{s_1} \rightarrow \dot{H}^{1 - s_2}} \nonumber \\
& \leq & C \, \lambda^{-1 + s_2} \norm{M^\lambda}{\dot{H}^{s_1} \rightarrow \dot{H}^{- s_2}}  \nonumber \\
& \leq & C \, \lambda^{-\tau}, \nonumber
\end{eqnarray}
and the proof is concluded.
\end{proof}

\begin{lemma} \label{twNorm}
Let $F, V \in \dot{H}^{s}$ where $0 < s < 1$. Then there exists a constant~$C$ such that 
\begin{equation}
\sup_{x \in \Omega} \abs{T^\lambda_{w}[F](x)} \leq C \lambda^{-s} \norm{F}{\dot{H}^{s}} \norm{V}{\dot{H}^{s}} \nonumber
\end{equation}
when $\lambda$ is sufficiently large.
\end{lemma}

\begin{proof}
Using Lemma \ref{mNorm} we obtain
\begin{eqnarray}
\abs{T^\lambda_{w}[F](x)} & \leq & C \, \lambda \norm{M^\lambda[F]}{\dot{H}^{-s}} \norm{w}{\dot{H}^{s}} \nonumber \\
& \leq & C \, \lambda^{1-s} \norm{F}{\dot{H}^{s}} \norm{(I - S_V^\lambda)^{-1} S_1^\lambda [V]}{\dot{H}^{s}}. \nonumber
\end{eqnarray}
As $(I - S_V^\lambda)^{-1}$ is bounded for $\lambda$ sufficiently large
(see \cite[Lemma 2.3]{AFR13}), 
\begin{eqnarray}
\abs{T^\lambda_{w} [F](x)} & \leq & C \lambda^{1-s} \norm{F}{\dot{H}^{s}} \norm{S_1^\lambda [V]}{\dot{H}^{s}} \nonumber \\
& \leq & C \lambda^{-s} \norm{F}{\dot{H}^{s}} \norm{V}{\dot{H}^{s}}, \nonumber
\end{eqnarray}
where the last inequality comes from applying Lemma \ref{sNorm}.	
\end{proof}

Let $g : \mathbb{R} \rightarrow \mathbb{R} \in C^n$. We say that
$x_s$ is a stationary point of $g$ of order $m < n$ 
if $g^{(k)}(x_s) = 0$ for $1 \leq k \leq m$ and 
$g^{(m+1)}(x_s) \neq 0$. We say that $g$ has stationary points if such points exist. The following lemma describes the asymptotic behaviour of 
a one dimensional oscillatory integral with a $C^2$ phase when 
there are only a finite number of stationary points of order one.

\begin{lemma} \label{statPhase1d}
Let $h \in W^{1,1}([a, b])$ and $g \in C^{2}([a, b])$ be such that
$g$ has only a finite number of stationary points of order at most one. Then there exists a constant $C$, independent of $h$ and depending
continuously on  the $C^2$ norm of $g$, such that
$$ \left \vert \int_a^b e^{i \lambda g(x)} h(x) dx \right \vert \leq C \lambda^{-1/2}\norm{h}{W^{1,1}([a,b])}$$
for $\lambda > 1$.
\end{lemma}

\begin{proof}
Let  $\lbrace s_j \rbrace_{j=1}^N$ denote the stationary points and $\delta = \min_j (\vert g''(s_j) \vert)$. Let $\epsilon$ 
be such that $\vert g''(x) \vert > \delta/2$ 
for all $x \in \cup_j U_j$,
where $U_j = (u_j^d, u_j^u) = B_\epsilon(s_j) \cap [a,b]$. 
Then we can make use of a version of
Van der Corput's lemma (see \cite[Corollary 2.6.8]{G08}), to obtain
$$ \left \vert \int_{U_j} e^{i \lambda g(x)} h(x) dx \right \vert \leq 24 \left( \frac{\delta}{2} \right)^{-1/2} \lambda^{-1/2} \left( \abs{h(u_j^u)} + \int_{U_j} \abs{h'(x)} dx \right). $$ 

Now let $V_j = (v_j^d, v_j^u)$ denote each of the remaining segments of $[a,b]$, 
such that $\cup_j V_j = [a,b] \setminus \cup_j U_j$.
Integrating by parts in each $V_j$ we obtain
\begin{eqnarray}
\int_{V_j} e^{i \lambda g(x)} h(x) dx 
&=& \frac{1}{i \lambda} \left[ \frac{e^{i \lambda g(x)} h(x)}{g'(x)} \right]_{v_j^d}^{v_j^u} \nonumber \\
&&- \frac{1}{i \lambda} \int_{V_j} e^{i \lambda g(x)} \frac{d}{dx} \left( \frac{h(x)}{g'(x)} \right) dx. \nonumber
\end{eqnarray} 
By Sobolev embedding \cite[Theorem 4.12, Part 1, Case A]{AF03} 
we have that $\norm{h}{L^\infty [a,b]} \leq C\norm{h}{W^{1,1}[a,b]}$.
Making use of this and H\"older's inequality, altogether  we obtain
\begin{eqnarray}
\left \vert \int_a^b e^{i \lambda g(x)} h(x) dx \right \vert & \leq & 48 N \left( \frac{\delta}{2}\right)^{-1/2} \lambda^{-1/2}\norm{h}{W^{1,1}([a,b])}  \nonumber \\
&& + (N+1) \ \kappa \, \lambda^{-1}\norm{h}{W^{1,1}([a,b])}, \nonumber
\end{eqnarray}
where 
$$\kappa = \max_j \left\lbrace 2 \norm{(g')^{-1}}{L^\infty(V_j)} + \norm{\frac{g' - g''}{(g')^2}}{L^\infty(V_j)} \right\rbrace,$$ 
which is finite, as there are no stationary points in $\cup_j V_j$.
For $\lambda > 1$ we can take
$$C = 48 N \left( \frac{\delta}{2}\right)^{-1/2} + (N+1) \, \kappa. $$
As $\delta$ and $\kappa$ depend continuously on the $C^2$ 
norm of $g$, the proof is concluded.
\end{proof}

\subsection{Piecewise $W^{s, 1}$-potentials}

We say that a curve $\mathcal{C}$ in the plane is contained in 
$C^{m, \alpha}$, with $m\ge 0$ and $ 0 \leq \alpha \leq 1$, 
if there exists a finite collection of bounded open sets $\lbrace U_j \rbrace_{j=1}^N$ 
such that $\mathcal{C} \subset \cup_{j=1}^N U_j$, and functions $f_j \in C^{m, \alpha}(\mathbb{R})$ 
such that
$$ \mathcal{C} \cap U_j \subset \{ (x, f_j(x))\,:\, x\in\mathbb{R} \} \quad \text{or} \quad \mathcal{C} \cap U_j \subset \{(f_j(y),y)\,:\, y\in\mathbb{R}\}. $$
For $\alpha = 0$, we write $f \in C^{m,0}(\mathbb{R})$ 
whenever $f$ and its derivatives up to order~$m$ are continuous and bounded; occasionally we will describe the curve as being simply $C^m$.

Consider two curves $\mathcal{C}_1$ and $\mathcal{C}_2$ for which there is
a finite cover by open sets $\lbrace U_j \rbrace_{j=1}^N$ 
such that for each $j$ either
$$ \mathcal{C}_1 \cap U_j \subset \{ (x, f_{1,j}(x))\,:\, x\in\mathbb{R} \} \quad \text{and} \quad \mathcal{C}_2 \cap U_j \subset \{(x, f_{2,j}(x))\,:\, x\in\mathbb{R}\}$$
or 
$$ \mathcal{C}_1 \cap U_j \subset \{ (f_{1,j}(y), y)\,:\, y\in\mathbb{R} \} \quad \text{and} \quad \mathcal{C}_2 \cap U_j \subset \{(f_{2,j}(y), y)\,:\, y\in\mathbb{R}\}.$$
Then we define the distance between the two curves in $C^{m,\alpha}$ norm as
$$ d(\mathcal{C}_1, \mathcal{C}_2) = \inf \left \lbrace \sup_j \left \lbrace \norm{f_{1,j} - f_{2,j}}{C^{m,\alpha}} \right \rbrace \right \rbrace $$
where the infimum is taken over all possible common covers. 
If a common cover does not exists for the curves,
we write $d(\mathcal{C}_1, \mathcal{C}_2) = \infty$.

A bounded Lipschitz domain whose boundary is a finite union of $C^{m, \alpha}$ curves will be referred to as a \textit{piecewise $C^{m, \alpha}$ domain}. The potentials that we will consider
exhibit discontinuities over $C^{2,\alpha}$ curves, where $1/2 < \alpha \leq 1$. 
More precisely, we are concerned with \textit{piecewise $W^{s, 1}$-potentials} $V$
that can be expressed as 
\begin{equation}
V(x) = \sum_{j=1}^N q_j(x) \goodchi_{\Omega_j} (x), \nonumber
\end{equation}
where $q_j \in W^{s, 1}(\mathbb{R}^2)$ and $\Omega_j$ 
are piecewise $C^{2,\alpha}$ domains with $1/2 < \alpha \leq 1$. 
We will use the following norm for these potentials:
$$ \norm{V}{D^{s,r}} = \inf \left \lbrace \sum_{j=1}^N \norm{q_j}{W^{s,1}} \left( 1 + \norm{\goodchi_{\Omega_j}}{H^{r}} \right) : V(x) = \sum_{j=1}^N q_j(x) \goodchi_{\Omega_j} (x) \right\rbrace $$
where $q_j$ and $\Omega_j$ are as previously described.

The following lemma provides a bound for the $L^2$ Sobolev norm 
for the potentials of our interest.

\begin{lemma} \label{weakDerivatives}
Let $q: \mathbb{R}^2 \rightarrow \mathbb{R}$ and let $\Omega$ be a 
bounded Lipschitz domain in the plane. 
Then there exists a constant $C$ independent of $q$ and $\Omega$ such that

i) For $0<r$ we have 
$$ \norm{q}{H^{r}} \leq C \norm{q}{W^{r + 1,1}}. $$

ii) For $0 < r < 1/2$ we have
$$\norm{q \, \goodchi_{\Omega}}{H^{r}} \leq C \norm{q\, \goodchi_{\Omega}}{D^{2,r}}.$$
\end{lemma}

\begin{proof}

Let $m$ be the largest integer less than $r$, let $t = r - m$ 
and let $D^t=~(-\Delta)^{t/2}$. For the first case we 
can use Sobolev embedding; see for example \cite[Theorem 4.12, Part 1, Case C]{AF03},
and the fact that $W^{r+1,1} \hookrightarrow W^{m+1,1}$ to obtain
\begin{eqnarray}
\norm{q}{H^{r}} &\leq& \left( \norm{q}{H^m}^2 + \norm{D^t q}{H^m}^2 \right)^{1/2} 
\leq C \left( \norm{q}{H^m}^2 + \norm{D^t q}{W^{m+1,1}}^2 \right)^{1/2} \nonumber \\
&\leq & C \left( \norm{q}{H^m}^2 + \norm{q}{W^{r+1,1}}^2 \right)^{1/2} 
\leq C \norm{q}{W^{r+1,1}}. \nonumber
\end{eqnarray}

For the second case we can use the generalised Leibniz rule \cite[Theorem A.12]{KPV93}, which states that
$$\norm{D^r(q \, \goodchi_{\Omega}) - q \, D^r(\goodchi_{\Omega}) - \goodchi_{\Omega} \, D^r(q)}{L^2} \leq C \norm{q}{L^\infty} \norm{D^r(\goodchi_{\Omega})}{L^2},$$
and so by triangle inequality we obtain
\begin{equation}
\norm{q \, \goodchi_{\Omega}}{H^{r}} \leq C \left( \norm{q}{L^\infty} \norm{\goodchi_{\Omega}}{H^{r}} + \norm{q}{H^{r}} \norm{\goodchi_{\Omega}}{L^\infty} \right).  \label{liebnizType}
\end{equation}
For the first term in the right-hand side, we can use
Sobolev embedding; see for example \cite[Theorem 4.12, Part 1, Case A]{AF03}, to obtain
$$ \norm{q}{L^\infty} \leq C \norm{q}{W^{2, 1}}. $$
On the other hand, we have $\goodchi_{\Omega} \in H^{r}$ 
for $r < 1/2$; see for example~\cite{FR13}. For the second term in the right-hand side of \eqref{liebnizType} 
we can use case $i)$, combined with the embedding $W^{2,1} \hookrightarrow W^{1+r,1}$, concluding the proof.
\end{proof}

\subsection{A topological property of $C^1$ graphs}

Finally, we provide a simple continuity result that will be useful
for characterizing the topological properties of the set of 
points where the reconstruction is not guaranteed 
as well as the continuity properties of the error bound of the reconstruction.

\begin{lemma} \label{implicitComplement}
Let $f \in C^{1}[a,b]$ and $\lbrace x_j \rbrace_{j=1}^N \in (a,b)$ 
be such that $f(x_j) = 0$, $f'(x_j) \neq 0 $ and 
$f(x) \neq 0$ for all $x \in [a,b] \setminus \lbrace x_j \rbrace_{j=1}^N$.
Then, for any $\epsilon
 > 0$, there exists $\delta
$ such that if 
$\norm{f - g}{C^1} < \delta
$ then there exists 
$\lbrace x_j^* \rbrace_{j=1}^N \in (a,b)$ such that
$$\vert x_j - x_j^* \vert < \epsilon
,$$
with $g(x_j^*) = 0$, $g'(x_j^*) \neq 0$ and 
$g(x) \neq 0$ for all  $x \in [a,b] \setminus \lbrace x_j^* \rbrace_{j=1}^N$.
\end{lemma}

\begin{proof}
Let $\eta = \min_j \lbrace \vert f'(x_j) \vert \rbrace$
and let $r$ be such that $\vert f'(x) \vert > \eta / 2$ for all  $x \in \cup B_r(x_j)$.
Let 
$$a = \inf_{x \in [a,b] \setminus \cup B_r(x_j)} \lbrace \vert f(x) \vert \rbrace.$$ 
Then, for all $g$ such that 
$$\norm{f - g}{C^1} < \tfrac{1}{2}\min \lbrace a, \eta/2 \rbrace,$$
we have 
\begin{equation}
\vert g(x) \vert > a / 2, \quad \forall\ x \in [a,b] \setminus \cup_j B_r(x_j) \nonumber
\end{equation}
and 
\begin{equation}
\vert g'(x) \vert > \eta / 4,\quad \forall\ x \in \cup_j B_r(x_j). \nonumber
\end{equation}

Now, as $\norm{f - g}{C^0} < a /2,$ we know that 
for all $x \in [a,b] \setminus \cup B_r(x_j)$ 
we have $f(x)g(x) > 0$ ($f$ and $g$ have the same sign
outside the balls $B_r(x_j)$), and so, by the intermediate
value theorem, $g$ must vanish
in each of the balls $B_r(x_j)$.
The fact that $g$ only vanishes at a single point $x_j^*$
in each of the balls is a consequence of the fact that $g$
is monotonous inside them. As $\vert g'(x) \vert > \eta / 4$
inside the balls, then, whenever
$$\norm{f - g}{C^0} < \epsilon
 \eta / 4,$$
we have that
$$\vert x_j - x_j^* \vert < \epsilon
.$$
Taking 
$$\delta
 = \min \lbrace a / 2, \eta / 4, \epsilon
 \eta / 4 \rbrace$$
concludes the proof.
\end{proof}

\section{Recovery}\label{rec}

Later we will see that recovery is not guaranteed, even at points that lie far from the discontinuities of the potential. In order to bound the measure of these points, we will require the following key lemmas.

\begin{lemma} \label{tangentLinesArea}
Let 
$\mathcal{C}$ be a $C^1$ curve contained in a bounded planar domain $\Omega$. 
Then, the union of tangent lines to $\mathcal{C}$ with a fixed slope $s$ has
zero Lebesgue measure in $\Omega$.
\end{lemma}

\begin{proof}
Let $\mathcal{C}=\{(x,\Gamma(x)) \}$, $\epsilon > 0$ and 
$$ I_{\epsilon} = \lbrace x \, : \,\Gamma(x)\in\Omega,\ \vert \Gamma'(x) - s \vert < \epsilon  \rbrace. $$
As $\Gamma \in C^1$,  $I_{\epsilon}$ is open, and as the real line satisfies the 
countable chain condition, $I_{\epsilon}$ must consist of a  countable union of 
disjoint intervals $\lbrace U_j \rbrace$.
For $x_0, x_1 \in U_j, x_0 < x_1$, by the Fundamental Theorem of Calculus, we have
$$ \Gamma(x_1) = \Gamma(x_0) + \int_{x_0}^{x_1} \Gamma'(x)\, dx. $$
As  $\vert \Gamma'(x) - s \vert < \epsilon$ 
in the domain of integration, then
\begin{equation}
(x_1 - x_0)(s - \epsilon) < \Gamma(x_1) - \Gamma(x_0) < (x_1 - x_0)(s + \epsilon), \nonumber
\end{equation}
and
\begin{equation}
\vert \Gamma(x_1) - \Gamma(x_0) - s (x_1 - x_0) \vert < \epsilon \, (x_1 - x_0). \label{maxDistance}
\end{equation}

Now let $l_{x_0,s}$ denote the line-segment with slope $s$
that contains the point $(x_0, \Gamma(x_0))$;
$$
l_{x_0,s}=\{(x,y) \in \Omega : y = \Gamma(x_0)+s(x-x_0)) \}.
$$
We write
$p_0 = (x_0, \Gamma(x_0)) \in l_{x_0,s}$ and $p_1 = (x_1, \Gamma(x_1)) \in l_{x_1,s}$.
As $p_0 \in l_{x_0,s}$, then so is $p_0^* = (x_1, \Gamma(x_0) + s(x_1 - x_0))$.
By \eqref{maxDistance} we know that $d(p_0^*, p_1) < \epsilon (x_1 - x_0)$ 
and so it follows that
 $L_j = \bigcup_{x \in U_j} l_{x,s} $
 is contained in a rectangle of width bounded by
$\epsilon \, |U_j| $ and length bounded by $\mathrm{diam}(\Omega)$.
Thus, 
\begin{eqnarray}
\Big|\bigcup_j L_j \Big| 
\leq \sum_j \left| L_j \right| < \sum_j  \epsilon \, |U_j| \, \mathrm{diam}(\Omega)
\leq \, \epsilon \,  \mathrm{diam}(\Omega)^2. \nonumber
\end{eqnarray}
Letting $\epsilon$ tend to zero, the proof is complete. 
\end{proof}

\begin{lemma} \label{noRecoveryZeroMeasure}
Let $\Omega$ be a bounded domain in the plane, let
$\mathcal{C}$ be a curve contained in $\Omega$ 
which is the graph of a $C^{2,\alpha}$ function with $1/2 < \alpha \leq 1$, 
and let
$\phi_x(z) = (x_1 - z_1)^2 - (x_2 - z_2)^2$.
Then the set of points $x \in \Omega$ such that 
either

i) $\phi_x \vert_\mathcal{C}$  has an infinite number
of stationary points,

ii) $\phi_x \vert_\mathcal{C}$ has at least one stationary point of order greater
than one, \\
has zero Lebesgue measure and is closed.
\end{lemma}

\begin{proof}
First we see that if $\phi_x \vert_\mathcal{C}$  has an infinite number
of stationary points, it has at least one stationary point of order greater
than one, and so the first case is contained in the second.
Let $x$ be such that $\phi_x\vert_\mathcal{C}$ has an 
infinite number of stationary points.
Then, by the compactness of 
$ \mathcal{C} $, there exists a sequence of 
stationary points $\lbrace z_{1,i} \rbrace_{i=1}^\infty$ and a point $z_{1,\infty}$ such that
$$\lim_{i \rightarrow \infty} z_{1,i} = z_{1,\infty}$$
with 
$$\frac{\partial \phi_x\vert_\mathcal{C}}{\partial z_1}(z_{1,\infty}) = 0.$$
As $\phi_x\vert_\mathcal{C}$ is a $C^2$ function and 
$\frac{\partial \phi_x\vert_\mathcal{C}}{\partial z_1}$ vanishes 
at all the stationary points then we have
$$\frac{\partial^2 \phi_x\vert_\mathcal{C}}{\partial z_1^2}(z_{1,\infty}) = \lim_{i \rightarrow \infty} \frac{\frac{\partial \phi_x\vert_\mathcal{C}}{\partial z_1}(z_{1,i+1}) - \frac{\partial \phi_x\vert_\mathcal{C}}{\partial z_1}(z_{1,i}) }{z_{1,i+1} - z_{1,i}} = 0,$$
and therefore $z_{1,\infty}$ is a stationary point of order greater than one.

Now we see that the set of $x$ 
such that $\phi_x\vert_\mathcal{C}$
has a stationary point of order greater than one
is null.
That is, the set of $x$ such that
$$ \frac{\partial \phi_x\vert_\mathcal{C}}{\partial z_1}  = \frac{\partial^2 \phi_x\vert_\mathcal{C}}{\partial z_1^2}  = 0 $$
has zero measure. Letting $\Gamma \in C^{2, \alpha}$ be such that 
$\mathcal{C}=\{(z_1,\Gamma(z_1)) \}$, the previous condition 
can be written as
\begin{equation}
z_1 - x_1 - \Gamma'(z_1) (\Gamma(z_1) - x_2) = 1 - \Gamma''(z_1) (\Gamma(z_1) - x_2) - \Gamma'(z_1)^2 = 0 \nonumber
\end{equation}
leading to
\begin{eqnarray}
x_1 = z_1 + \Gamma'(z_1) (x_2 - \Gamma(z_1)), \label{counterTangent} \\
\Gamma''(z_1) \left( x_2 - \Gamma(z_1) \right) =  \Gamma'(z_1)^2 - 1. \label{height} 
\end{eqnarray}

First we consider the case where $|\Gamma''(z_1)|>\delta>0$. As $\Gamma''$ is continuous and the real line satisfies the countable chain condition,  for this to be satisfied $z_1$ must lie in one of at most countably many intervals. Taking one such interval $U$ and rearranging \eqref{counterTangent} and \eqref{height}, the set of $x$ such that $ \phi_x\vert_\mathcal{C}$ has a stationary point
of order greater than one at $z_1 \in U$ is given by
\begin{equation}
x = (x_1, x_2) = G(z_1) = \left( z_1 + \frac{\Gamma'(z_1)^3 - \Gamma'(z_1)}{\Gamma''(z_1)}, \Gamma(z_1)  + \frac{\Gamma'(z_1)^2 - 1}{\Gamma''(z_1)} \right). \nonumber
\end{equation}
We see that the set of $x$ is the image of a $C^{0,\alpha}$ function. To see that such a set has 
zero measure, take $\lbrace U_j \rbrace_{j=1}^{2N}$ 
a covering of $U$ such that $\vert U_j\vert = |U| / N.$ 
Now as
$G(U_j)$ is contained in a ball of radius $\lesssim(|U| / N)^{\alpha},$ 
we obtain
$\vert G(U) \vert \lesssim |U|^{2\alpha} N^{1 - 2\alpha} $. 
As $\alpha > 1/2$, we can let $N$ tend to infinity to conclude that $|G(U)|=0$. This is the only place where we require the H\"older regularity. Now as the countable union of null sets is null, we have concluded the proof in this case.

On the other hand, if $\vert \Gamma''(z_1) \vert \le \delta$ 
and $x \in \Omega$, then by \eqref{height} it follows that 
$\Gamma'(z_1)^2$ must be contained in the interval $[1-\delta^*, 1+\delta^*]$
with $\delta^* = \delta\text{diam}(\Omega)$. Let
$l_{z_1,s}$ be the line that passes through $(z_1,\Gamma(z_1))$ with slope $s$,
and let
$$T(\delta^*) = \bigcup l_{z_1, 1 / \Gamma'(z_{1})}, \ \forall\ z_1 : \Gamma'(z_1)^2 \in [ 1 - \delta^*, 1 + \delta^* ].$$
From equation \eqref{counterTangent} we see that the remaining set 
of $x$ such that $\phi_x \vert_\mathcal{C}$ has a stationary point 
of order greater than one at $(z_1,\Gamma(z_1))$ is contained in 
$T(\delta^*)$.
Using Lemma \ref{tangentLinesArea}, we have
\begin{align*}\lim_{n \rightarrow \infty} \left| T(1/n) \cap \Omega \right| &=\left| \bigcap_{n=1}^\infty \left(T(1/n) \cap \Omega \right) \right| =\left|T(0) \cap \Omega \right|\\
&=\left|\bigcup_{z_1: \Gamma'(z_1)=1} l_{z_1, 1}\cap\Omega\right|+\left|\bigcup_{z_1: \Gamma'(z_1)=-1} l_{z_1, -1}\cap\Omega\right|=0.
\end{align*}
Therefore, for any $\varepsilon > 0 $ we can take $\delta$ small enough so that
$\left| T(\delta^*) \cap \Omega \right| < ~ \varepsilon,$
allowing us to conclude that the set of points $x$ such that $\phi_x\vert_\mathcal{C}$
has a stationary point of order greater than one
is null.

To see that the set is closed, first notice that 
for any $\delta > 0$, there exists $r>0$ such that for any $x' \in B_r(x)$ we have
$$\norm{\phi_x\vert_\mathcal{C} - \phi_{x'}\vert_\mathcal{C}}{C^2} < \delta.$$ Thus, applying Lemma \ref{implicitComplement} to 
$\frac{\partial \phi_x\vert_\mathcal{C}}{\partial z_1}$ 
we see that whenever
$\phi_x\vert_\mathcal{C}$ has a finite number of stationary points of 
degree at most one then 
$\phi_{x'}\vert_\mathcal{C}$ will have the same number of stationary points
and of the same degree for any $x'$ close enough to $x$.
This means that the set of points $x$ such that $\phi_x\vert_\mathcal{C}$ has a finite
number of stationary points of degree at most one is open, and the complement
is closed, concluding the proof.
\end{proof}

Suppose that $0$ is not a Dirichlet eigenvalue for the Hamiltonian $V - \Delta$. Then, for each $f \in H^{1/2}$ there exists a unique solution $u \in H^1$ to \eqref{dirichletProblem}, and the DtN map can be defined by 
\begin{equation}
\langle \Lambda_V[f], v\vert_{\partial \Omega} \rangle = \int_{\partial \Omega} \Lambda_V [f] \, v\vert_{\partial \Omega}  = \int_\Omega V \, u \, v + \nabla u \cdot \nabla v \label{DtNDef}
\end{equation}
for any $v \in H^1(\Omega)$.

Theorem~\ref{reconstructionTheorem0} is contained in the following result in which we also obtain decay rates for potentials with slightly more regularity.

\begin{theorem}\label{reconstructionTheorem}
Let $V$ be a piecewise $W^{s,1}$-potential, with $0 < s-2 <  2r < 1$. Then
for almost every $x \in \Omega$, there exists a constant 
$C_x = C(x, \cup \partial \Omega_j)$ such that
\begin{equation}
\abs{ \frac{\lambda}{\pi} \int_{\partial \Omega} e^{i \lambda \overline{\psi}} \left( \Lambda_V - \Lambda_0 \right) \left[ u_{\lambda,x} \right] - V(x) } \leq C_x \lambda^{1-s/2} \left( \norm{V}{D^{s,r}} + \norm{V}{D^{s,r}}^2 \right) \nonumber
\end{equation}
whenever $\lambda$ is sufficiently large. Moreover, if $s = 2$ then we have
\begin{equation}
\lim_{\lambda \rightarrow \infty } \frac{\lambda}{\pi} \int_{\partial \Omega} e^{i \lambda \overline{\psi}} \left( \Lambda_V - \Lambda_0 \right) \left[ u_{\lambda,x} \right] = V(x),\quad \text{a.e.}\ x\in\Omega. \nonumber
\end{equation}
\end{theorem}

\begin{proof}
As the DtN is a self-adjoint operator
and $e^{i \lambda \overline{\psi_x}}$ satisfies the Laplace equation,
then we can use the DtN map definition \eqref{DtNDef} to see that
\begin{eqnarray}
\frac{\lambda}{\pi} \int_{\partial \Omega} e^{i \lambda \overline{\psi_x}} \left( \Lambda_V - \Lambda_0 \right) \left[ u \right] 
= \frac{\lambda}{\pi} \int_\Omega e^{i \lambda \phi_x} \, V \, (1 + w), \nonumber
\end{eqnarray}
where $\phi_x = \psi_x + \overline{\psi_x}$. 
Recalling that $0 \leq s-2 <  2r < 1$, by Lemma \ref{twNorm} we have
$$ \sup_{x \in \Omega} \left\vert \frac{\lambda}{\pi} \int_\Omega e^{i \lambda \phi_x} \, V \, w \right\vert \leq C \, \lambda^{-r} \ \norm{V}{H^r}^2, \quad \text{as } \lambda \rightarrow \infty, $$
and by part $ii)$ of Lemma~\ref{weakDerivatives} this yields
\begin{equation}
\sup_{x \in \Omega} \left\vert \frac{\lambda}{\pi} \int_\Omega e^{i \lambda \phi_x} \, V \, w \right\vert \leq C \, \lambda^{1-s/2} \ \norm{V}{D^{s,r}}^2, \quad \text{as } \lambda \rightarrow \infty. \label{smallW}
\end{equation}

Now, by Lemma \ref{noRecoveryZeroMeasure} we know that for
almost every $x$ in $\Omega$, $\phi_x\vert_{\cup \partial \Omega_j}$
has only a finite number of stationary points of order at most one. We now 
prove that the reconstruction formula recovers the potential correctly at these points for piecewise $W^{s,1}$-potentials, $s>2$ (almost all of them for $W^{2,1}$-potentials).
First we split the integral 
\begin{equation}
\frac{\lambda}{\pi}\int_\Omega e^{i \lambda \phi_x} \, V \, = \frac{\lambda}{\pi}\sum_{j=1}^N \int_{\Omega_j} e^{i \lambda \phi_x} \, q_j, \label{split}
\end{equation}
where we write $\phi$ for $\phi_x$ from now on.
We will prove 
that the value of each of these integrals tends to zero sufficiently fast whenever 
the integration domain does not contain $x$, the point at which we are reconstructing. Then we show that the value of the integral that contains $x$ 
converges to $V(x)$. 
Without loss of generality, we can suppose that $x$ belongs to the interior of~$\Omega_1$. For $j > 1$, we use Green's first identity, 
with $u = \frac{e^{i \lambda \phi}}{i \lambda}$ and 
$\nabla v =~\frac{q_j \, \nabla \phi}{\vert \vert \nabla \phi \vert \vert^2}$, 
to obtain
\begin{equation}
 \int_{\Omega_j} e^{i \lambda \phi} \, q_j = \frac{1}{i \lambda} \int_{\partial\Omega_j} e^{i \lambda \phi} \frac{q_j \, \nabla \phi}{\enorm{\nabla \phi}^2} \cdot \mathbf{n} - \frac{1}{i \lambda} \int_{\Omega_j} e^{i \lambda \phi} \, \nabla \cdot \left( \frac{q_j \, \nabla \phi}{\enorm{\nabla \phi}^2} \right). \nonumber
\end{equation}
Using Green's first identity again on the second term with 
$u = \frac{e^{i \lambda \phi}}{i \lambda}$ and 
${ \nabla v = \nabla \cdot \left( \frac{q_j \, \nabla \phi}{\vert \vert \nabla \phi \vert \vert^2} \right) \frac{\nabla \phi}{\vert \vert \nabla \phi \vert \vert^2} }$ 
leads to
\begin{eqnarray}
\frac{\lambda}{\pi}\int_{\Omega_j} e^{i \lambda \phi} \, q_j  &=& \frac{1}{i \pi} \int_{\partial\Omega_j} e^{i \lambda \phi} \, \frac{q_j \, \nabla \phi}{\enorm{\nabla \phi}^2} \cdot \mathbf{n} \nonumber \\
&& + \frac{1}{\pi\lambda} \int_{\partial\Omega_j} e^{i \lambda \phi} \, \diver \left(\frac{q_j \, \nabla \phi}{\enorm{\nabla \phi}^2} \right) \frac{\nabla \phi}{\enorm{\nabla \phi}^2} \cdot \mathbf{n} \label{byParts} \\
&& - \frac{1}{\pi\lambda} \int_{\Omega_j} e^{i \lambda \phi} \, \diver \left( \diver \left(\frac{q_j \, \nabla \phi}{\enorm{\nabla \phi}^2} \right) \frac{\nabla \phi}{\enorm{\nabla \phi}^2} \right). \nonumber
\end{eqnarray}
As the number of stationary points
on $\partial \Omega_j$ is finite and are of order at most one, and
by trace theorem we know that $q_j\vert_{\partial \Omega_j} \in W^{1,1}(\partial \Omega_j)$
(see for example \cite[Section 5.5, Theorem 1]{E10}), 
then we can use
Lemma \ref{statPhase1d} on each of the $C^2$ components of 
$\partial \Omega_j$, together with H\"older's inequality, to see that there exists 
$C_x' = C(x, \partial \Omega_j)$ such that
\begin{equation}
\left\vert \int_{\partial \Omega_j} e^{i \lambda \phi} \frac{q_j \nabla \phi}{\enorm{\nabla \phi}^2} \cdot \mathbf{n} \right\vert \leq C_x' \lambda^{-1/2}  \norm{\frac{\nabla \phi \cdot \mathbf{n}}{\enorm{\nabla \phi}^2}}{W^{1,\infty} (\partial \Omega_j)} \norm{q_j}{W^{1,1}(\partial \Omega_j)}, \label{auxConstant}
\end{equation}
as $ \lambda \rightarrow \infty$. As $x$ belongs to the interior  of 
$\Omega_1$, we have that 
$\enorm{\nabla \phi}^{-2}$ is bounded, and
using the trace theorem this yields to
\begin{equation}
\Big|\int_{\partial\Omega_j} e^{i \lambda \phi} \, \frac{q_j \, \nabla \phi}{\enorm{\nabla \phi}^2} \cdot \mathbf{n}\Big| \leq C_x \lambda^{-1/2} \norm{V}{D^{s,r}}, \quad \text{as } \lambda \rightarrow \infty. \label{part1}
\end{equation}
For the second term on the right-hand side of \eqref{byParts} 
we can use H\"older's inequality to obtain
\begin{eqnarray}
&& \left\vert \int_{\partial\Omega_j} e^{i \lambda \phi} \, \diver \left(\frac{q_j \, \nabla \phi} {\enorm{\nabla \phi}^2} \right) \frac{\nabla \phi}{\enorm{\nabla \phi}^2} \cdot \mathbf{n} \right\vert \nonumber \\
&& \leq \norm{\diver \frac{q_j \, \nabla \phi}{\enorm{\nabla \phi}^2} }{L^1(\partial\Omega_j)} 
\norm{\frac{\nabla \phi  \cdot \mathbf{n}}{\enorm{\nabla \phi}^2}}{L^\infty(\partial\Omega_j)} \nonumber \\
&& \leq \left( \norm{\phi}{\dot{W}^{1,\infty}(\partial\Omega_j)} \norm{\enorm{\nabla \phi}^{-2}}{L^\infty(\partial\Omega_j)} + \norm{\diver \frac{\nabla \phi}{\enorm{\nabla \phi}^2}}{L^\infty(\partial\Omega_j)} \right) \nonumber \\
&& \ \ \ \times \norm{\frac{\nabla \phi  \cdot \mathbf{n}}{\enorm{\nabla \phi}^2}}{L^\infty(\partial\Omega_j)} \norm{q_j}{W^{1,1}(\partial \Omega_j)}, \nonumber
\end{eqnarray} 
and by the trace theorem we get
\begin{equation}
\left\vert \int_{\partial\Omega_j} e^{i \lambda \phi} \, \diver \left(\frac{q_j \, \nabla \phi}{\enorm{\nabla \phi}^2} \right) \frac{\nabla \phi}{\enorm{\nabla \phi}^2} \cdot \mathbf{n} \right\vert
\leq C_x \norm{V}{D^{s,r}}. \label{part2}
\end{equation}
Similarly, for the last term  on the right-hand side of \eqref{byParts} we have
\begin{eqnarray}
&& \left\vert \int_{\Omega_j} e^{i \lambda \phi} \, \diver \left( \diver \left(\frac{q_j \, \nabla \phi}{\enorm{\nabla \phi}^2} \right) \frac{\nabla \phi}{\enorm{\nabla \phi}^2} \right) \right\vert \nonumber \\
&& \leq \norm{\diver \frac{q_j \nabla \phi}{\enorm{\nabla \phi}^2}}{\dot{W}^{1,1}(\Omega_j)} \norm{\phi}{\dot{W}^{1,\infty}(\Omega_j)}  \norm{\enorm{\nabla \phi}^{-2}}{L^\infty(\Omega_j)} \nonumber \\
&& \ \ \ +  \norm{\diver \frac{q_j \nabla \phi}{\enorm{\nabla \phi}^2}}{L^1(\Omega_j)} \norm{\diver \frac{\nabla \phi}{\enorm{\nabla \phi}^2}}{L^\infty (\Omega_j)} \nonumber 
\end{eqnarray}
\begin{eqnarray}
&& \leq \left( \norm{\phi}{W^{2,\infty}(\Omega_j)} \norm{\enorm{\nabla \phi}^{-2}}{W^{1,\infty}(\Omega_j)} + \norm{\diver \frac{\nabla \phi}{\enorm{\nabla \phi}^2}}{\dot{W}^{1,\infty}(\Omega_j)} \right) \nonumber \\
&& \ \ \ \times \norm{\phi}{\dot{W}^{1,\infty}(\Omega_j)}  \norm{\enorm{\nabla \phi}^{-2}}{L^\infty(\Omega_j)} \norm{q_j}{W^{2,1}(\Omega_j)} \nonumber \\
&& \ \ \ +  \left( \norm{\phi}{\dot{W}^{1,\infty}(\Omega_j)} \norm{\enorm{\nabla \phi}^{-2}}{L^\infty(\Omega_j)} + \norm{\diver \frac{\nabla \phi}{\enorm{\nabla \phi}^2}}{L^\infty(\Omega_j)} \right) \nonumber \\
&& \ \ \ \times  \norm{\diver \frac{\nabla \phi}{\enorm{\nabla \phi}^2}}{L^\infty (\Omega_j)} \norm{q_j}{W^{1,1}(\Omega_j)} \nonumber \\
&& \leq \left( \norm{\phi}{W^{2,\infty}(\Omega_j)} \norm{\enorm{\nabla \phi}^{-2}}{W^{1,\infty}(\Omega_j)} + \norm{\diver \frac{\nabla \phi}{\enorm{\nabla \phi}^2}}{W^{1,\infty}(\Omega_j)} \right)^2 \nonumber \\
&& \ \ \ \times \norm{q_j}{W^{2,1}(\Omega_j)} \nonumber 
\end{eqnarray} 
yielding
\begin{eqnarray} 
\left\vert \int_{\Omega_j} e^{i \lambda \phi} \, \diver \left( \diver \left(\frac{q_j \, \nabla \phi}{\enorm{\nabla \phi}^2} \right) \frac{\nabla \phi}{\enorm{\nabla \phi}^2} \right) \right\vert \leq C_x \norm{V}{D^{s,r}}. \label{part3}
\end{eqnarray} 
Plugging \eqref{part1}, \eqref{part2} and \eqref{part3} into \eqref{byParts} we obtain
\begin{equation}
\abs{ \frac{\lambda}{\pi}\int_{\Omega_j} e^{i \lambda \phi} \, q_j} \leq C_x \lambda^{-1/2} \norm{V}{D^{s,r}} \quad \text{as } \lambda \rightarrow \infty. \label{nonStationaryOmega}
\end{equation}

We now consider $\int_{\Omega_1} e^{i \lambda \phi}q_1$ by decomposing $q_1$ into
\begin{equation}
q_x = q_1 \chi, \quad q_{\mathrm{rem}} = q_1 (1 - \chi), \nonumber
\end{equation}
where $\chi(z)$ is a bump function such that 
\begin{equation}
\chi (z) = \left\lbrace \begin{array}{rl}
1 & \text{if } \ \norm{z - x}{}    \leq r_1, \\
0 & \text{if } \ \norm{z - x}{}     \geq r_2,
\end{array} 
\right.\nonumber
\end{equation}
with $0 < r_1 < r_2 < d(x, \partial \Omega_1)$. As $q_{\mathrm{rem}}(y) = 0$ for $y$ close enough to $x$, we 
can use the same arguments that lead to \eqref{nonStationaryOmega} 
to obtain
\begin{equation}
\abs{ \frac{\lambda}{\pi} \int_{\Omega_1} e^{i \lambda \phi} \, q_{\mathrm{rem}}} \leq C_x \lambda^{-1/2} \norm{q_{\mathrm{rem}}}{W^{2,1}}, \quad \text{as } \lambda \rightarrow \infty. \label{stationaryOmegaRest}
\end{equation}
On the other hand, as $q_1 \in W^{2,1}(\Omega_1)$, we can use
Sobolev embedding (see for example \cite[Theorem 4.12, Part 1, Case C]{AF03})
to see that $q_x \in H^1_0(\Omega_1)$. Now it was noted in \cite{AFR13} that 
$\frac{\lambda}{\pi} \int_{\Omega_1} e^{i \lambda \phi} \, q_x$ can be interpreted as the solution to a nonelliptic {\it time dependent} Schr\"odinger equation, at time $1/\lambda$. Thus, using the almost everywhere convergence result of \cite[Theorem 1]{RVV} we obtain
\begin{equation}
\lim_{\lambda \rightarrow \infty} \frac{\lambda}{\pi} \int_{\Omega_1} e^{i \lambda \phi} \, q_x = q_x(x)=q_1(x)=V(x) \quad \text{a.e.}\ x\in \mathbb{R}^2.\label{stationaryOmega}
\end{equation}
If $q_1 \in W^{s,1}$ with $s > 2$, then we can recover at all the remaining points and we get a decay rate. Indeed, 
\begin{align*}
\left|\frac{\lambda}{\pi} \int_{\Omega_1} e^{i \lambda \phi} \, q_x - q_x(x)\right|
&= \left|\frac{\lambda}{\pi} \left( q_x * e^{i \lambda (z_1^2 - z_2^2)} \right) (x) - q_x(x) \right| \\
&= \left|\frac{1}{4\pi^2} \int e^{ix\cdot\xi} \, \widehat{q}_x(\xi) \, \left( e^{-i \frac{1}{\lambda} (\xi_1^2 - \xi_2^2)} - 1 \right) d\xi\right| \\
&\leq \norm{q_x}{\dot{H}^{s-1}} \left( \int \frac{\left| e^{-i \frac{1}{\lambda} (\xi_1^2 - \xi_2^2)} - 1 \right|^2}{|\xi|^{2s-2}} d\xi \right)^{1/2} \\
&= \norm{q_x}{\dot{H}^{s-1}} \lambda^{1-s/2} \left( \int \frac{2 - 2 \cos (\xi_1^2 - \xi_2^2)}{|\xi|^{2s-2}} d\xi \right)^{1/2} \\
&= \norm{q_x}{\dot{H}^{s-1}} \lambda^{1-s/2} \left( \int \frac{\sin^2 (\frac{1}{2}(\xi_1^2 - \xi_2^2))}{|\xi|^{2s-2}} d\xi \right)^{1/2}.
\end{align*}
Note that  for $2<s<3$ we have
$$\lim_{\abs{\xi} \rightarrow 0} \frac{\sin (\xi_1^2 - \xi_2^2)}{\abs{\xi}^{s-1}} \leq \lim_{\abs{\xi} \rightarrow 0} \frac{\sin (\abs{\xi}^2)}{\abs{\xi}^2} = 1$$
thus the integral is finite, and we can use 
part $i)$ of Lemma \ref{weakDerivatives} to obtain
\begin{equation}
\abs{ \frac{\lambda}{\pi} \int_{\Omega_1} e^{i \lambda \phi} \, q_x - q_x(x)} \leq C \lambda^{1-s/2} \norm{q_x}{W^{s,1}}, \quad \text{as } \lambda \rightarrow \infty. \label{stationaryOmegaRate}
\end{equation}

Plugging \eqref{nonStationaryOmega}, \eqref{stationaryOmegaRest}, \eqref{stationaryOmega}, and \eqref{stationaryOmegaRate} into \eqref{split}, 
together with \eqref{smallW} concludes the proof.
\end{proof}

\begin{remark}
As noted in \cite{AFR13}, $\frac{\lambda}{\pi} \int_{\Omega} e^{i \lambda \phi} \, V$ can be interpreted 
as the solution to a nonelliptic {\it time dependent} Schr\"odinger equation 
at time $1/\lambda$. Therefore equations \eqref{nonStationaryOmega}, \eqref{stationaryOmegaRest}, \eqref{stationaryOmega}, and \eqref{stationaryOmegaRate}
imply almost everywhere convergence to the initial data $V$, 
whenever $V$ is piecewise-$W^{s, 1}$ with $2 \leq s < 3$.
\end{remark}

A consequence of Theorem \ref{reconstructionTheorem} is that potentials of this type can also be 
recovered from the scattering data at a fixed energy.

\begin{corollary} 
Let $V$ be a piecewise $W^{2,1}$-potential. Then $V$ can be recovered 
almost everywhere from the scattering amplitude at a fixed energy~$k$. 
\end{corollary}

\begin{proof}
Let $Q$ be a square such that $\cup_{j=1}^N \Omega_j \subset Q$. 
In \cite{AFR14} expressions are given for computing $\Lambda_{V-k^2}$
defined on $\partial Q$ from the scattering amplitude at energy~$k^2$. 
Therefore, the recovery from the scattering amplitude follows directly from the fact
that if $V$ is piecewise $W^{2,1}$-potential, then so is 
$V -~k^2 \goodchi_Q$, which allows us to recover the potential using 
Theorem \ref{reconstructionTheorem}.
\end{proof}

For the stability estimates of the sequel we will require some continuity properties of the constant 
in Theorem \ref{reconstructionTheorem} which we record as a lemma.

\begin{lemma} \label{continuousConstant}
Let $\mathcal{N}$ be the set of $x$ such that $\phi_{x}\vert_{\cup \partial \Omega_j}$ 
has a stationary point of degree greater that one.
Then the constant $C_x = C(x, \cup \partial \Omega_j)$ 
in Theorem~\ref{reconstructionTheorem}
has the following 
continuity properties in $\Omega \setminus \mathcal{N}$:

i) It is continuous with respect to $x$.

ii) It is continuous with respect to 
$\cup \partial \Omega_j$ in the $C^2$ norm.
\end{lemma}

\begin{proof}
Let $N, \, \Omega_j$ and $r_1$ be as in the proof of 
Theorem \ref{reconstructionTheorem}.
Let $\Omega_j^* = \Omega_j$ for $j =2,..., N$ and let 
$$\Omega_1^* = \Omega_1 \setminus B_{r_1}(x).$$
The constant $C_x$ is given by

\begin{eqnarray}
C_x &=& \sum_{j=1}^N C_x' \norm{\frac{\nabla \phi \cdot \mathbf{n}}{\enorm{\nabla \phi}^2}}{W^{1,\infty} (\partial \Omega_j^*)} + \norm{\frac{\nabla \phi  \cdot \mathbf{n}}{\enorm{\nabla \phi}^2}}{L^\infty(\partial\Omega_j^*)} \nonumber \\
&& \times \left( \norm{\phi}{\dot{W}^{1,\infty}(\partial\Omega_j^*)} \norm{\enorm{\nabla \phi}^{-2}}{L^\infty(\partial\Omega_j^*)} + \norm{\diver \frac{\nabla \phi}{\enorm{\nabla \phi}^2}}{L^\infty(\partial\Omega_j^*)} \right) \nonumber \\
&& + \left( \norm{\phi}{W^{2,\infty}(\Omega_j^*)} \norm{\enorm{\nabla \phi}^{-2}}{W^{1,\infty}(\Omega_j^*)} + \norm{\diver \frac{\nabla \phi}{\enorm{\nabla \phi}^2}}{W^{1,\infty}(\Omega_j^*)} \right)^2. \nonumber
\end{eqnarray} 
The constant $C_x'$ appears in equation \eqref{auxConstant} 
by the use of Lemma \ref{statPhase1d} (taking $g=\phi\vert_{\partial \Omega_j}$
and $h=\frac{q_j \nabla \phi \cdot \mathbf{n}}{\enorm{\nabla \phi}^2}\big\vert_{\partial\Omega_j}$),
so it is continuous with respect to $x$ and 
with respect to $\cup \partial \Omega_j$ in the $C^2$ norm,
and as $\nabla \phi$ does not vanish inside any of the $\Omega_j^*$,
then so is $C_x$, concluding the proof.
\end{proof}

As is to be expected, the error in the reconstruction 
increases the closer we move to the discontinuities of the potential,
as the constant $C_x$ blows up, and we are unable to recover at the discontinuities.
It is perhaps more interesting that, for certain potentials,
there are points where the reconstruction fails which are far
from the discontinuities of the potential.

Indeed, consider the potential given by $V =  \goodchi_{\Omega_1}$, where
$\Omega_1$ is the rhombus with vertices at $(0,0), (1, 1), (2, 0)$
and $(1, -1)$; see Figure 1. Consider the problem of recovering the potential inside 
$\Omega = [-2,2] \times [-2,2]$. We might expect to be able to recover at the points $x=(-t,-t)$, for $t \in (0, 2)$, far from the potential.
However, by Alessandrini's identity 
\cite[Lemma 1]{A88}, we know that the reconstructed potential $\tilde{V}$ at $x$ is the limit as $\lambda$ tends to infinity of  
\begin{equation}
\frac{\lambda}{\pi} \int_{\partial \Omega} e^{i \lambda \overline{\psi_x}} \left( \Lambda_V - \Lambda_0 \right) \left[ u_{\lambda,x} \right] = \frac{\lambda}{\pi} \int_\Omega e^{i \lambda \phi_x} \, V \, (1 + w_{\lambda,x}), \nonumber
\end{equation}
where $\phi_x = \psi_x + \overline{\psi}_x$, which can be rewritten as
\begin{equation}
\frac{\lambda}{\pi} \left( \int_\Omega e^{i \lambda \phi_x} \, V \, w_{\lambda,x} + \int_{\Omega \setminus \Omega_1} e^{i \lambda \phi_x} \, V + \int_{\Omega_1} e^{i \lambda \phi_x} \, V  \right). \nonumber 
\end{equation}
For the first term we can use Lemma \ref{twNorm} and part $ii)$ of Lemma \ref{weakDerivatives} to obtain
\begin{equation}
\left\vert \frac{\lambda}{\pi} \int_\Omega e^{i \lambda \phi_x} \, V \, w_{\lambda,x} \right\vert \leq C \, \lambda^{-r} \ \norm{V}{D^{2,r}}^2, \quad \text{as } \lambda \rightarrow \infty, \nonumber
\end{equation}
for any $0 < r < 1/2$. As the potential is equal to 1 inside $\Omega_1$
and zero in the rest of the domain, this yields
\begin{equation}
\tilde{V}(x) = \lim_{\lambda\to\infty}\frac{\lambda}{\pi} \int_{\Omega_1} e^{i \lambda \phi_x}. \nonumber 
\end{equation}
Using Green's first identity twice we get
\begin{eqnarray}
\frac{\lambda}{\pi}\int_{\Omega_1} e^{i \lambda \phi_x} &=& \frac{1}{i \pi} \int_{\partial\Omega_j} e^{i \lambda \phi_x} \, \frac{\nabla \phi_x}{\enorm{\nabla \phi_x}^2} \cdot \mathbf{n} \nonumber \\
&& + \frac{1}{\pi\lambda} \int_{\partial\Omega_1} e^{i \lambda \phi_x} \, \diver \left(\frac{\nabla \phi_x}{\enorm{\nabla \phi_x}^2} \right) \frac{\nabla \phi_x}{\enorm{\nabla \phi_x}^2} \cdot \mathbf{n} \nonumber \\
&& - \frac{1}{\pi\lambda} \int_{\Omega_1} e^{i \lambda \phi_x} \, \diver \left( \diver \left(\frac{\nabla \phi_x}{\enorm{\nabla \phi_x}^2} \right) \frac{\nabla \phi_x}{\enorm{\nabla \phi_x}^2} \right). \nonumber
\end{eqnarray}
As we have seen in the proof of Theorem \ref{reconstructionTheorem}, the second and third terms converge to zero as $\lambda$ tends to infinity, and for the first term we can write
\begin{eqnarray}
\int_{\partial\Omega_j} e^{i \lambda \phi_x} \, \frac{\nabla \phi_x}{\enorm{\nabla \phi_x}^2} \cdot \mathbf{n} &=& \int_{l_1} e^{i \lambda \phi_x} \, \frac{\nabla \phi_x}{\enorm{\nabla \phi_x}^2} \cdot \mathbf{n_1}
+ \int_{l_2} e^{i \lambda \phi_x} \, \frac{\nabla \phi_x}{\enorm{\nabla \phi_x}^2} \cdot \mathbf{n_2} \nonumber \\
&&+ \int_{l_3} e^{i \lambda \phi_x} \, \frac{\nabla \phi_x}{\enorm{\nabla \phi_x}^2} \cdot \mathbf{n_3}
+ \int_{l_4} e^{i \lambda \phi_x} \, \frac{\nabla \phi_x}{\enorm{\nabla \phi_x}^2} \cdot \mathbf{n_4} \nonumber
\end{eqnarray}
where
\begin{eqnarray}
l_1 &=& (s, s) \text{ for } s \in (0, 1), \nonumber \\
l_2 &=& (1 + s, 1 - s) \text{ for } s \in (0, 1), \nonumber \\
l_3 &=& (2 - s, -s) \text{ for } s \in (0, 1), \nonumber \\
l_4 &=& (1 - s, s - 1) \text{ for } s \in (0, 1). \nonumber
\end{eqnarray}
As we have $\phi(z) = (z_1 + t)^2 - (z_2 + t)^2$, we see that
\begin{eqnarray}
\phi_x|_{l_1} (s) &=& 0, \nonumber \\
\phi_x|_{l_2} (s) &=& 4 s (t + 1), \nonumber \\
\phi_x|_{l_3} (s) &=& 4 (t - s + 1), \nonumber \\
\phi_x|_{l_4} (s) &=& 4 t (1 - s). \nonumber
\end{eqnarray}
Therefore, we can apply Lemma~\ref{statPhase1d} to three of the sides;
$$ \int_{l_j} e^{i \lambda \phi_x} \, \frac{\nabla \phi_x}{\enorm{\nabla \phi_x}^2} \cdot \mathbf{n_j} = O(\lambda^{-1/2}), \quad \text{for } j = 2, 3, 4,$$
and on the other hand we have
$$ \int_{l_1} e^{i \lambda \phi_x} \, \frac{\nabla \phi_x}{\enorm{\nabla \phi_x}^2} \cdot \mathbf{n_1} = \int_0^1  \frac{-\sqrt{2}}{4 (s + t)} \, ds = \frac{\sqrt{2}}{4} \left( \log(t) - \log(t + 1) \right). $$
Putting everything together we obtain
$$ \tilde{V}(x) = \frac{\sqrt{2}i}{4\pi} \log(1+ 1/t)\neq 0.$$

\begin{figure}
\centering
\includegraphics[width=8cm]{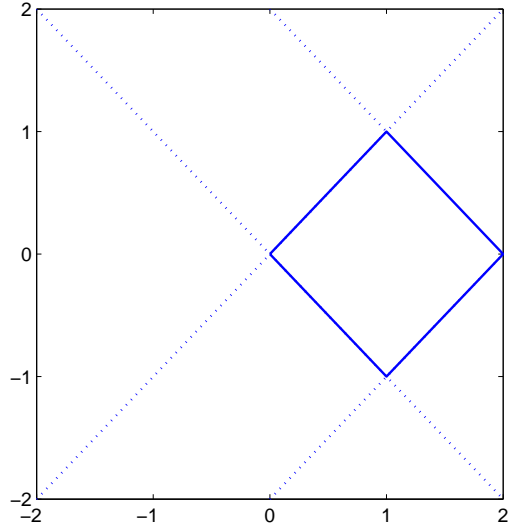}
\vspace{-0.1cm}
\caption{
Solid lines are the discontinuities of the potential and dashed lines are points far from the discontinuities where the recovery fails.
}
\end{figure}

\section{Stability} \label{stabilitySection}

We begin with some preliminary results that we will require for the proof of the stability estimates.

\begin{lemma} \label{twwNorm}
Let $V_1, V_2 \in H^{s}$, where $0 < 2 s \leq 1$ and let $F \in L^2$. Then there exists a constant $C$ such that 
\begin{equation}
\sup_{x \in \Omega} \abs{T^\lambda_{w_1 w_2}[F](x)} \leq C \lambda^{-2s} \norm{F}{L^2} \norm{V_1}{\dot{H}^{s}} \norm{V_2}{\dot{H}^{s}} \nonumber
\end{equation}
when $\lambda$ is sufficiently large.
\end{lemma}

\begin{proof}
By H\"older's inequality and the Hardy-Littlewood-Sobolev lemma we get
\begin{eqnarray}
\abs{T^\lambda_{w_1 w_2} [F](x)} 
& \leq & \lambda \norm{F w_1 w_2}{L^1} \nonumber \\
& \leq & \lambda \norm{F}{L^2} \norm{w_1}{L^4} \norm{w_2}{L^4} \nonumber \\
& \leq & C \lambda \norm{F}{L^2} \norm{w_1}{\dot{H}^{1/2}} \norm{w_2}{\dot{H}^{1/2}}. \nonumber
\end{eqnarray}
As $(I - S_V^\lambda)^{-1}$ is bounded for $\lambda$ sufficiently large
(see \cite[Lemma 2.3]{AFR13}), this yields
\begin{eqnarray}
\abs{T^\lambda_{w_1 w_2} [F](x)} & \leq & C \lambda \norm{F}{L^2} \norm{S_1^\lambda [V_1]}{\dot{H}^{1/2}} \norm{S_1^\lambda [V_2]}{\dot{H}^{1/2}}. \nonumber
\end{eqnarray}
Applying Lemma \ref{sNorm} twice concludes the 
proof.
\end{proof}

We will also require the following crude bound for Bukhgeim solutions.

\begin{lemma} \label{uBound}
Let $0<r<1/2$, and let $V$ be a piecewise $W^{2,1}$-potential contained in a bounded planar domain $\Omega$ with diameter $d$. Then there exists a constant $C$ depending on  
$\Omega$ such that the Bukhgeim solutions satisfy
$$ \norm{u_{\lambda,x}}{H^1(\Omega)} \leq C e^{\lambda d^2} \left( 1 + \norm{V}{D^{2,r}}^2 \right) $$
whenever $\lambda$ is sufficiently large.
\end{lemma}

\begin{proof} Writing $u=u_{\lambda,x}$, 
\begin{eqnarray}
\norm{u}{H^1(\Omega)} &\leq& C \left( \norm{u}{L^2(\Omega)} + \norm{\partial_{\overline{z}}u}{L^2(\Omega)}  + \norm{\partial_z u}{L^2(\Omega)} \right) \nonumber \\
&=& C \norm{e^{i \lambda \psi} (1+w)}{L^2(\Omega)} + C \norm{\partial_{z}^{-1} V e^{i \lambda \psi} (1+w) }{L^2(\Omega)}  \nonumber \\
&& +\, C \norm{\partial_{\overline{z}}^{-1} V e^{i \lambda \psi} (1+w) }{L^2(\Omega)}  \nonumber
\end{eqnarray}
Note that $\partial_{z}^{-1}$ and $\partial_{\overline{z}}^{-1}$ 
are bounded operators, as $V$ has compact support (see for example \cite[Theorem 4.3.12]{AIM08}). As we have 
$\norm{\cdot}{L^2(\Omega)} \leq C \norm{\cdot}{L^4(\Omega)}$, 
then using H\"older's inequality leads to
\begin{eqnarray}
\norm{u}{H^1(\Omega)} &\leq& C \norm{e^{i \lambda \psi}}{L^\infty(\Omega)}  \left( 1 + \norm{V}{L^4} \right) \norm{1 + w}{L^4(\Omega)} \nonumber \\
&\leq& C e^{i \lambda d^2} \left( 1 + \norm{V}{L^4} \right) \left( 1 + \norm{w}{L^4(\Omega)} \right), \nonumber 
\end{eqnarray}
and by the Hardy-Littlewood-Sobolev inequality we get
\begin{eqnarray}
\norm{u}{H^1(\Omega)} &\leq& C e^{i \lambda d^2} \left( 1 + \norm{V}{L^4} \right) \left( 1 + \norm{w}{\dot{H}^{1/2}} \right). \nonumber 
\end{eqnarray}
As $(I - S^\lambda_{V})^{-1}$ is bounded for $\lambda$ sufficiently large 
(see \cite[Lemma~2.3]{AFR13}), we have
\begin{equation}
\norm{u}{H^1(\Omega)} \leq C e^{i \lambda d^2} \left( 1 + \norm{V}{L^4} \right) \left( 1 + \norm{S_1^\lambda [V]}{\dot{H}^{1/2}} \right). \nonumber
\end{equation}
Now, by Lemma \ref{sNorm} we get
\begin{equation}
\norm{u}{H^1(\Omega)} \leq C e^{i \lambda d^2} \left( 1 + \norm{V}{L^4} \right) \left( 1 + \norm{V}{\dot{H}^{r}} \right). \nonumber
\end{equation}
Using Sobolev embedding (see for example \cite[Theorem 4.12, Part 1, Case~A]{AF03}) we have
$$ \norm{V}{L^4} \leq C \sum_{j=1}^N \norm{q_j}{L^4} \leq \norm{V}{D^{2,r}}.$$
Using part $ii)$ of Lemma~\ref{weakDerivatives} Lemma concludes the proof.
\end{proof}

We now prove an a priori stability estimate for reconstruction from 
the DtN map in the $L^\infty$ norm. Note that the result is interesting
when there is a priori knowledge of where, approximately, the 
discontinuities lie, as the constant term depends on the point under
consideration with respect to the discontinuities.
The result has been stated in the following form as in practical situations 
one could consider where a potential might lie, given a noisy reconstruction 
of it. If the assumption that the discontinuities of the potentials
are closed to each other was dropped, the constant $C_x$ would depend on
both $\lbrace \partial \Omega_{1,j} \rbrace$ and $\lbrace \partial \Omega_{2,j} \rbrace$, the discontinuities of $V_1$ and $V_2$ respectively.

\begin{theorem} \label{dtnStability}
Let $0 < s-2 <  2r < 1$ and let $V_1, V_2$ be piecewise $W^{s,1}$-potentials 
supported on a Lipschitz domain $\Omega$ in $\mathbb{R}^2$
such that their discontinuities are close enough with respect to the $C^2$ norm.
Then, for almost every $x \in \Omega$, there exists a constant $C_x = C(x, \cup \partial \Omega_{1,j})$
such that
\begin{equation}
\abs{V_1(x) - V_2(x)} \leq C_x \big|\ln \enorm{\Lambda_{V_1} - \Lambda_{V_2}} \big|^{1-s/2} ( \kappa + \kappa^3 )  \nonumber
\end{equation}
whenever $\Lambda_{V_1}$ and $\Lambda_{V_2}$ are close enough, where $\kappa = \max \lbrace \norm{V_1}{D^{s,r}}, \norm{V_2}{D^{s,r}} \rbrace$.
\end{theorem}

\begin{proof}
Let $d$ be the diameter of $\Omega$ and let
\begin{equation}
\lambda = \frac{-1}{6 d^2} \ln \enorm{\Lambda_{V_1} - \Lambda_{V_2}}. \label{lambdaValue}
\end{equation} 
Note that whenever $\Lambda_{V_1}$ and $\Lambda_{V_2}$ are sufficiently close, 
then $\lambda$ can be as large as we need. 

By the triangle inequality,
\begin{eqnarray}
\ \ \abs{V_1(x) - V_2(x)} &\leq& \abs{V_1(x) - T^\lambda_{1+w_1} [V_1](x)} + \abs{V_2(x) - T^\lambda_{1+w_2} [V_2](x)} \label{initial2} \\
&&+ \abs{T^\lambda_{1+w_1} [V_1](x) - T^\lambda_{1+w_2}[V_2](x)}. \nonumber
\end{eqnarray}
We can use Theorem \ref{reconstructionTheorem} on the first two terms to obtain
\begin{equation}
\abs{V_j(x) - T^\lambda_{1+w_j} [V_j](x)} \leq C_x \, \lambda^{1-s/2} (\kappa + \kappa^2) \label{statPhase2}
\end{equation}
where $C_x = C(x, \cup \partial \Omega_{1,j})$. We can take the same 
constant $C_x$ for both terms as it
is continuous with respect to the discontinuities in the $C^2$ norm
(due to Lemma \ref{continuousConstant}).
For the last term we have
\begin{eqnarray}
\abs{T^\lambda_{1+w_1} [V_1](x) - T^\lambda_{1+w_2}[V_2](x)} &\leq& \abs{T^\lambda_{(1+w_1)(1+w_2)}[V_1 - V_2](x)} \nonumber \\
&& + \abs{T^\lambda_{w_1}[V_2](x)} + \abs{T^\lambda_{w_2}[V_1](x)} \nonumber \\
&& + \abs{T^\lambda_{w_1 w_2}[V_1 - V_2](x)}. \nonumber
\end{eqnarray}
By Lemma \ref{twNorm} and part $ii)$ of Lemma \ref{weakDerivatives} we obtain
\begin{equation}
\norm{T^\lambda_{w_2}[V_1]}{L^\infty} \leq C \, \lambda^{1-s/2} \kappa^2 \quad \text{and} \quad \norm{T^\lambda_{w_1}[V_2]}{L^\infty} \leq C \, \lambda^{1-s/2} \kappa^2. \label{tCrossBound}
\end{equation}
and by Lemma \ref{twwNorm} and part $ii)$ of Lemma \ref{weakDerivatives} we obtain
\begin{equation}
\norm{T^\lambda_{w_1 w_2}[V_1 - V_2]}{L^\infty} \leq C \, \lambda^{1-s/2} \kappa^3. \label{tProdBound}
\end{equation}
Let $u_j$ be Bukhgeim solutions to $\Delta u_j = V_j u_j$. Then we have
\begin{eqnarray}
\norm{T^\lambda_{(1+w_1)(1+w_2)}[V_1 - V_2]}{L^\infty} &=& \frac{\lambda}{\pi} \norm{\int_\Omega (V_1 - V_2) u_1 u_2}{L^\infty} \nonumber \\
&=& \frac{\lambda}{\pi} \norm{\int_{\partial \Omega} f_1 \Lambda_{V_2} [f_2] - f_2 \Lambda_{V_1} [f_1]}{L^\infty}, \nonumber
\end{eqnarray}
where $f_j = u_j \vert_{\partial \Omega}$. As the 
DtN is a self-adjoint operator, we have that
\begin{equation}
\frac{\lambda}{\pi} \norm{\int_{\partial \Omega} f_1 \Lambda_{V_2} [f_2] - f_2 \Lambda_{V_1} [f_1]}{L^\infty} = \frac{\lambda}{\pi} \norm{\int_{\partial \Omega} \left( \Lambda_{V_1} - \Lambda_{V_2} \right) [f_1] f_2}{L^\infty}.  \nonumber
\end{equation}
For $f \in H^{1/2}$ the DtN map satisfies $\Lambda_v [f] \in H^{-1/2}$, where
$H^{-1/2}$ is the dual of $H^{1/2}$. Thus, for any $x \in \Omega$, we have
\begin{eqnarray}
\frac{\lambda}{\pi} \left\vert \int_{\partial \Omega} \left( \Lambda_{V_1} - \Lambda_{V_2} \right) [f_1] f_2 \right\vert &\leq& \lambda
\enorm{\Lambda_{V_1} - \Lambda_{V_2}} \norm{f_1}{H^{1/2}(\partial \Omega)} \norm{f_2}{H^{1/2}(\partial \Omega)} \nonumber \\
&\leq& \lambda \enorm{\Lambda_{V_1} - \Lambda_{V_2}} \norm{u_1}{H^1(\Omega)} \norm{u_2}{H^1(\Omega)} \nonumber
\end{eqnarray}
and we can use Lemma \ref{uBound} to obtain
\begin{equation}
\frac{\lambda}{\pi} \left\vert \int_{\partial \Omega} \left( \Lambda_{V_1} - \Lambda_{V_2} \right) [f_1] f_2 \right\vert \leq C \lambda e^{2 \lambda d^2} \enorm{\Lambda_{V_1} - \Lambda_{V_2}} \left( 1 + \kappa^4 \right). \label{boundaryIntegralBound2}
\end{equation}
Inserting \eqref{statPhase2}, \eqref{tCrossBound}, \eqref{tProdBound}
and \eqref{boundaryIntegralBound2} into \eqref{initial2}, and noting
that $\lambda < e^{\lambda d}$ for $\lambda$ sufficiently large,
leads to
\begin{eqnarray}
\abs{V_1(x) - V_2(x)} \leq C_x \ \lambda^{1-s/2} \ (\kappa + \kappa^3) + \, C \enorm{\Lambda_{V_1} - \Lambda_{V_2}} e^{3 \lambda d^2} ( 1 + \kappa^4). \nonumber
\end{eqnarray}
Taking $\lambda$ as in \eqref{lambdaValue} 
we obtain
\begin{eqnarray}
\abs{V_1(x) - V_2(x)} &\leq& C_x \left( \frac{-1}{6 d^2} \ln \enorm{\Lambda_{V_1} - \Lambda_{V_2}} \right)^{1-s/2} (1 + \kappa^3)  \nonumber \\
&& + \, C  \enorm{\Lambda_{V_1} - \Lambda_{V_2}}^{1/2} ( 1 + \kappa^4) \nonumber
\end{eqnarray}
where the second term 
can be omitted for $\enorm{\Lambda_{V_1} - \Lambda_{V_2}}$ 
small enough, concluding the proof.
\end{proof}

We now provide a link to the scattering amplitude. We adapt the proof 
of Stefanov (see \cite{S90}) to the two-dimensional case. 
Due to the severe ill-posedness of the problem, we need  to introduce a norm 
for the scattering amplitude which penalizes the higher components 
of the frequency spectrum.
Let $V$ be a potential supported on the unit disk, then we define the norm 
for its scattering amplitude at a fixed energy $k^2$ as
$$ \norm{A_V}{k} = \left( \sum_{n,m \in \mathbb{Z}} \left( \frac{3 + 3 \abs{n}}{k} \right)^{2 \abs{n}} \left( \frac{3 + 3 \abs{m}}{k} \right)^{2 \abs{m}} \abs{ a^{(n,m)}_V }^2 \right)^{1/2}, $$
where $a^{(n,m)}_V$ are the Fourier coefficients of $A_V$
$$ A_V (\eta, \theta) = \sum_{n,m \in \mathbb{Z}} a^{(n,m)}_V e^{i n \eta + i m \theta}. $$
Before passing to the proof, we define the single layer potential operator
$$ \mathcal{S}_V [f] (x) = \int_{\partial \Omega} G_V(x, y) f(y) dy, $$
where $G_V$ is the outgoing Green's function which satisfies
$$(- \Delta + V - k^2) \, G_V(x, y) = \delta (x - y).$$ 

\begin{lemma} \label{scatToDtnStability}
Let $V_1, V_2$ be two potentials supported in the unit disk. Then there exists a constant
$C_k = C(k)$ such that
\begin{equation}
\norm{\Lambda_{V_1} - \Lambda_{V_2}}{H^{1/2}(\mathbb{S}^1) \rightarrow H^{-1/2}(\mathbb{S}^1)} \leq C_k \norm{A_{V_1} - A_{V_2}}{k}. \nonumber
\end{equation}
\end{lemma} 

\begin{proof}
Using Nachman's formula \cite[Theorem 1.6]{Na88} we have
\begin{eqnarray}
\Lambda_{V_1 - k^2} - \Lambda_{V_2 - k^2} &=& \mathcal{S}_{V_1}^{-1} - \mathcal{S}_{V_2}^{-1} \nonumber \\
&=& \mathcal{S}_{V_1}^{-1} \left( \mathcal{S}_{V_2} - \mathcal{S}_{V_1} \right) \mathcal{S}_{V_2}^{-1}. \nonumber
\end{eqnarray}
As $\mathcal{S}_V$ is a bounded and invertible mapping from $H^{-1/2}(\mathbb{S}^1)$ 
to $H^{1/2}(\mathbb{S}^1)$ (see \cite[Proposition A.1]{IN95}), we have
\begin{equation}
\norm{\Lambda_{V_1 - k^2} - \Lambda_{V_2 - k^2}}{H^{1/2}(\mathbb{S}^1) \rightarrow H^{-1/2}(\mathbb{S}^1)} 
\le C \norm{\mathcal{S}_{V_1} - \mathcal{S}_{V_2}}{H^{-1/2}(\mathbb{S}^1) \rightarrow H^{1/2}(\mathbb{S}^1)}. \nonumber
\end{equation}

Letting $B_x = (1 - \Delta_x)^{1/4}$, we write
\begin{eqnarray}
&& \norm{\mathcal{S}_{V_1} - \mathcal{S}_{V_2}}{H^{-1/2}(\mathbb{S}^1) \rightarrow H^{1/2}(\mathbb{S}^1)} \nonumber \\
&& = \sup_{\enorm{f} = 1} \norm{ \int_{\mathbb{S}^1} (G_{V_1} (x, y) - G_{V_2}(x, y)) f(y) \, dy }{H^{1/2}(\mathbb{S}^1)} \nonumber \\
&& = \sup_{\enorm{f} = 1} \left( \int_{\mathbb{S}^1} \left( B_x \int_{\mathbb{S}^1} (G_{V_1} (x, y) - G_{V_2}(x, y)) f(y) \, dy \right)^2 dx \right)^{1/2} \nonumber \\
&& = \sup_{\enorm{f} = 1} \left( \int_{\mathbb{S}^1} \left( \int_{\mathbb{S}^1} B_x (G_{V_1} (x, y) - G_{V_2}(x, y)) f(y) \, dy \right)^2 dx \right)^{1/2} \nonumber
\end{eqnarray}
by Pareseval's identity we have
\begin{equation}
= \sup_{\enorm{f} = 1} \left( \int_{\mathbb{S}^1} \left( \int_{\mathbb{S}^1} \left( B_y B_x (G_{V_1} (x, y) - G_{V_2}(x, y)) \right) B_y^{-1} f(y) dy \right)^2 dx \right)^{1/2} \nonumber
\end{equation}
using Minkowski's integral inequality we get
\begin{eqnarray}
&&\leq \sup_{\enorm{f} = 1} \int_{\mathbb{S}^1} \left( \int_{\mathbb{S}^1} \left( \left( B_y B_x (G_{V_1} (x, y) - G_{V_2}(x, y)) \right) B_y^{-1} f(y) \right)^2 dx \right)^{1/2} dy  \nonumber \\
&&= \sup_{\enorm{f} = 1} \int_{\mathbb{S}^1} \abs{B_y^{-1} f(y)} \left( \int_{\mathbb{S}^1} \left( B_y B_x (G_{V_1} (x, y) - G_{V_2}(x, y)) \right)^2 dx \right)^{1/2} dy  \nonumber 
\end{eqnarray}
and using the Cauchy-Schwarz inequality
\begin{eqnarray}
&& \leq \sup_{\enorm{f} = 1} \norm{f}{H^{-1/2}(\mathbb{S}^1)} \norm{G_{V_1} - G_{V_2}}{H^{1/2}(\mathbb{S}^1) \otimes H^{1/2}(\mathbb{S}^1)} \nonumber \\
&& = \norm{G_{V_1} - G_{V_2}}{H^{1/2}(\mathbb{S}^1) \otimes H^{1/2}(\mathbb{S}^1)}. \nonumber
\end{eqnarray}

From \cite[Theorem 2.2]{AFR14} we know that
\begin{eqnarray}
G_{V_1} (x, y) - G_{V_2} (x, y) &=& \sum_{n,m \in \mathbb{Z}} \frac{(-1)^n}{16} i^{n+m} \left( a^{(n,m)}_{V_1} - a^{(n,m)}_{V_2} \right) \nonumber \\ 
&& \times \, H^{(1)}_n(k \abs{x}) \, H^{(1)}_m(k \abs{y}) \, e^{i n \phi_x + i m \phi_y} \nonumber
\end{eqnarray}
where $H^{(1)}$ denotes the Hankel function of the first kind.
Now, using Parseval's identity and the bound for the Hankel function 
in \cite[Lemma 2.3]{AFR14}, 
there exists $C_k' = C(k)$ such that
\begin{eqnarray}
&& \norm{G_{V_1} - G_{V_2}}{H^{1/2}(\mathbb{S}^1) \otimes H^{1/2}(\mathbb{S}^1)}^2 \nonumber \\
&\leq& \sum_{n,m \in \mathbb{Z}} (1 + n^2)^{1/2} (1 + m^2)^{1/2} \abs{ a^{(n,m)}_{V_1} - a^{(n,m)}_{V_2} }^2 
\abs{H^{(1)}_n(k)}^2 \abs{H^{(1)}_m(k)}^2 \nonumber \\
&\leq& C_k' \sum_{n,m \in \mathbb{Z}} (1 + n^2)^{1/2} (1 + m^2)^{1/2} \abs{ a^{(n,m)}_{V_1} - a^{(n,m)}_{V_2} }^2
{\abs{n}!}^2 {\abs{m}!}^2 \left( \frac{3}{k} \right)^{2 \abs{n} + 2 \abs{m}} \nonumber \\
&\leq& C_k' \sum_{n,m \in \mathbb{Z}} \left( \frac{3 + 3 \abs{n}}{k} \right)^{2 \abs{n}} \left( \frac{3 + 3 \abs{m}}{k} \right)^{2 \abs{m}} 
\abs{ a^{(n,m)}_{V_1} - a^{(n,m)}_{V_2} }^2 \nonumber
\end{eqnarray}
concluding the proof.
\end{proof}

\begin{corollary} \label{scatteringStability}
Let $0 < s-2 <  2r < 1$ and let $V_1, V_2$ be two piecewise $W^{s,1}$ 
potentials  supported on a bounded domain in $\mathbb{R}^2$
such that their discontinuities are close enough in the $C^2$ norm.
Then, for almost every $x \in \Omega$, there exist constants
$C_x = C(x, \cup \partial \Omega_j), C_k = C(k),$ 
such that
\begin{equation}
\abs{V_1(x) - V_2(x)} \leq C_x \big| \ln \left( C_k \norm{A_{V_1} - A_{V_2}}{k} \right) \big|^{1-s/2} ( \kappa + \kappa^3 )  \nonumber
\end{equation}
whenever $\Lambda_{V_1}$ and $\Lambda_{V_2}$ are close enough, where $\kappa = \max \lbrace \norm{V_1}{D^{s,r}}, \norm{V_2}{D^{s,r}}\rbrace$.
\end{corollary}

\textbf{Acknowledgements:} 
The author thanks  Daniel Faraco and Keith Rogers 
for their valuable comments and corrections
and thanks Evgeny Lakshtanov for bringing \cite{LNV15} to his attention.
This work has been partially supported by ERC-277778 and by ICMAT Severo Ochoa project SEV-2015-0554 (MINECO).



\begin{thebibliography}{99}

\bibitem{AF03}
R. A. Adams and John J. F. Fournier.
{\em Sobolev Spaces}.
\newblock Pure and Applied Mathematics, Elsevier (2003).

\bibitem{A88}
G. Alessandrini.
{\em Stable determination of conductivity by boundary measurements}.
\newblock Applicable Analysis (1988).

\bibitem{A14}
G. Alessandrini.
{\em Global stability for a coupled physics inverse problem}.
\newblock Inverse Problems (2014).

\bibitem{AFR13}
K. Astala, D. Faraco and K. M. Rogers.
{\em Unbounded potential recovery in the plane}.
\newblock  Annales scientifiques de l'École Normale Supérieure, to appear, arXiv:1304.1317 (2015).

\bibitem{AFR14} 
K. Astala, D. Faraco and K. M. Rogers.
{\em Recovery of the Dirichlet-to-Neumann map from scattering data in the plane}.
\newblock RIMS Kokyuroku Bessatsu (2014).

\bibitem{AIM08}
K. Astala, T. Iwaniec and G. Martin.
{\em Elliptic Partial Differential Equations and Quasiconformal Mappings in the Plane}.
\newblock Princeton University Press (2008).

\bibitem{AP06} 
K. Astala and L. P\"{a}iv\"{a}rinta.
{\em Calder\'on Inverse Conductivity problem in plane}.
\newblock Annals of Mathematics (2006).

\bibitem{BBR01} 
J. A. Barcel\'o, T. Barcel\'o and A. Ruiz.
{\em Stability of the Inverse Conductivity Problem in the Plane for Less Regular Conductivities}.
\newblock Journal of Differential Equations (2001).

\bibitem{BFR07} 
T. Barcel\'o , D. Faraco and A. Ruiz.
{\em Stability of Calder\'on inverse conductivity problem in the plane}.
\newblock Journal de Mathématiques Pures et Appliquées (2007).

\bibitem{BIY15} 
E. Blasten, O. Yu. Imanuvilov and M. Yamamoto. 
{\em Stability and uniqueness for a two-dimensional inverse boundary 
value problem for less regular potentials}.
\newblock Inverse Problems and Imaging (2015).

\bibitem{BT03}
R. M. Brown and R. H. Torres.
{\em Uniqueness in the inverse conductivity problem for conductivities with 
 3/2 derivatives in $L^p, p > 2n$}.
\newblock Journal of Fourier Analysis and Applications (2003).

\bibitem{BU97}
R. M. Brown and G. A. Uhlmann.
{\em Uniqueness in the inverse conductivity problem for nonsmooth conductivities in two dimensions}.
\newblock Communications in Partial Differential Equations (1997).

\bibitem{B08}
A. L. Bukhgeim.
{\em  Recovering a potential from Cauchy data in the two-dimensional case}.
\newblock Journal of Inverse and Ill-Posed Problems (2008).

\bibitem{C80}
A. P. Calder\'on.
{\em  On an inverse boundary value problem}.
\newblock Seminar on Numerical Analysis and its Applications to Continuum Physics, Rio de Janeiro (1980).

\bibitem{CGR13} 
P. Caro, A. Garc\'ia and J. M. Reyes. 
{\em Stability of the Calder\'on problem for less regular conductivities}. 
Journal of Differential Equations (2013).

\bibitem{CR} 
P. Caro and K. M. Rogers. 
{\em Global uniqueness for the Calder\'on problem with Lipschitz conductivities}. 
Forum of Mathematics Pi (2016).

\bibitem{CFR10} 
A. Clop, D. Faraco and A. Ruiz. 
{\em Stability of Calder\'on's inverse conductivity problem in the plane for discontinuous conductivities}. 
Inverse Problems and Imaging (2010).

\bibitem{E10} 
L. C. Evans. 
{\em Partial Differential Equations}. 
Graduate Studies in Mathematics (2010).

\bibitem{FR13}
D. Faraco and K. M. Rogers.
\newblock {\em The Sobolev norm of characteristic functions with applications to the Calder\'on inverse problem}.
\newblock The Quarterly Journal of Mathematics (2013).

\bibitem{G08}
L. Grafakos.
\newblock {\em Classical Fourier analysis}.
\newblock Springer, Graduate Texts in Mathematics (2008).

\bibitem{H}
B.~Haberman. {\em Uniqueness in Calder\'on's problem for conductivities with unbounded gradient}. Communications on Mathematical Physics (2015).

\bibitem{HT}
B.~Haberman, D.~Tataru. 
{\em Uniqueness in Calder\'on's problem with Lipschitz conductivities}. Duke Mathematical Journal (2013).

\bibitem{IN95}
V. Isakov and A. I. Nachman.
{\em Global uniqueness for a two-dimensional semilinear elliptic inverse problem}.
\newblock Transactions of the American Mathematical Society (1995).

\bibitem{KPV93}
C. Kenig, G. Ponce and L. Vega.
{\em Well-posedness and scattering results for the generalized
  Korteweg-de Vries equation via the contraction principle}.
\newblock Communications in Pure and Applied Mathematics (1993).

\bibitem{KSU07}
C. Kenig, J. Sj\"ostrand and G. Uhlmann.
{\em The Calder\'on problem with partial data}.
\newblock Annals of Mathematics (2007).

\bibitem{LNV15}
E. L. Lakshtanov, R. G. Novikov and B. R. Vainberg.
{\em A global Riemann-Hilbert problem for two-dimensional inverse scattering at fixed energy}.
\newblock arXiv:1509.06495 (2015).

\bibitem{Na88}
A. I. Nachman.
{\em Reconstructions from boundary measurements}.
\newblock Annals of Mathematics (1988).

\bibitem{Na96}
A. I. Nachman.
{\em Global uniqueness for a two-dimensional inverse boundary value problem}.
\newblock Annals of Mathematics (1996).

\bibitem{No88}
R. G. Novikov.
{\em Multidimensional inverse spectral problem for the equation $-\Delta \psi + (v(x) - Eu(x)) \psi = 0$}.
\newblock Functional Analysis and Its Applications (1988).

\bibitem{NS10}
R. G. Novikov and M. Santacesaria.
{\em A global stability estimate for the Gel'fand-Calder\'on 
inverse problem in two dimensions}.
\newblock Journal of Inverse and Ill-Posed Problems (2010).

\bibitem{NS11}
R. G. Novikov and M. Santacesaria.
{\em Global uniqueness and reconstruction for the multi-channel Gel'fand-Calder\'on 
inverse problem in two dimensions}.
\newblock Bulletin des Sciences Math\'ematiques (2011).

\bibitem{RVV} K. M. Rogers, A. Vargas and L. Vega. 
{\em Pointwise convergence of solutions to the nonelliptic Schr\"odinger equation}.
\newblock Indiana University Mathematical Journal (2006).

\bibitem{S90}
P. Stefanov.
{\em Stability of the inverse problem in potential scattering at fixed energy}.
\newblock Annales de l'institue Fourier (1990).

\bibitem{SU87}
J. Sylvester and G. Uhlmann.
{\em A global uniqueness theorem for an inverse boundary value problem}.
\newblock Annals of Mathematics (1987).



\end{thebibliography}
\end{document}